\documentclass{proc-l}

\usepackage{thmtools}
\usepackage{thm-restate}

\usepackage[matrix,arrow,curve]{xy}
\usepackage{graphicx}
\usepackage[english]{babel}

\usepackage{mathtools}

\usepackage{color}

\newcommand{\Hom}{{\rm Hom}}
\newcommand{\Ext}{{\rm Ext}}

\newcommand{\Add}{{\rm Add}}

\newtheorem{theorem}{Theorem}[section]
\newtheorem{proposition}[theorem]{Proposition}
\newtheorem{corollary}[theorem]{Corollary}
\newtheorem{lemma}[theorem]{Lemma}

\theoremstyle{definition}
\newtheorem{definition}[theorem]{Definition}

\theoremstyle{remark}
\newtheorem{remark}[theorem]{Remark}

\newtheorem{setup}{Setup}

\numberwithin{equation}{section}

\begin{document}

\title{Derived equivalences induced by nonclassical tilting objects}


\author{Luisa Fiorot}
\address{Dipartimento di Matematica, Universit\`a degli studi di Padova, Via Trieste 63, 35121 Padova (Italy)}
\email{fiorot@math.unipd.it}
\thanks{}

\author{Francesco Mattiello}
\address{Dipartimento di Matematica, Universit\`a degli studi di Padova,  Via Trieste 63, 35121 Padova (Italy)}
\email{mattiell@math.unipd.it}
\thanks{Mattiello is supported by Assegno di ricerca ``Tilting theory in triagulated categories'' del Dipartimento di Matematica dell'Universit\`a degli Studi di Padova  and by Progetto di Eccellenza della fondazione Cariparo.}

\author{Manuel Saor\'in}
\address{Departamento de Matem\'aticas, Universidad de Murcia, Aptdo. 4021, 30100 Espinardo, Murcia (Spain)}
\email{msaorinc@um.es}
\thanks{Saor\'in is
supported by research projects from the Spanish Ministerio de Econom\'ia y Competitividad (MTM2013-46837-P)  and from the Fundaci\'on 'S\'eneca'
of Murcia (19880/GERM/15), with a part of FEDER funds.}
\thanks{The authors thank their institutions for their help. They also thank Jorge Vit\'oria and Alexandra Zvonareva for pointing out some omissions in a first draft of the paper.}

\subjclass[2010]{Primary 18E30; Secondary, 18E10, 18G55}

\date{}

\dedicatory{}


\begin{abstract}
Suppose that $\mathcal{A}$ is an abelian category whose derived category $\mathcal{D}(\mathcal{A})$ has $Hom$ sets and arbitrary (small) coproducts, let $T$  be a (not necessarily classical) ($n$-)tilting object of $\mathcal{A}$ and let $\mathcal{H}$ be the heart of the associated t-structure on $\mathcal{D}(\mathcal{A})$. We show 
that there is a triangulated equivalence of unbounded derived categories $\mathcal{D}(\mathcal{H})\stackrel{\cong}{\longrightarrow}\mathcal{D}(\mathcal{A})$ which is compatible with the inclusion functor $\mathcal{H}\hookrightarrow\mathcal{D}(\mathcal{A})$.
The result admits a straightforward dualization to cotilting objects in abelian categories whose derived category has $Hom$ sets and arbitrary products.
\end{abstract}

\maketitle


\section*{Introduction}

Tilting theory first appeared in the eighties of last century in the context of Representation Theory of finite dimensional algebras (see \cite{BB}, \cite{HR}, \cite{B}) as a generalization of Morita Theory.  Instead of the classical equivalence of module categories given by the theorem of Morita, for a tilting $A$-module $T$, which was always finite dimensional in these seminal papers,  the so-called Brenner-Butler theorem gave a counter-equivalence between the torsion pairs $(\text{Ker}(\text{Ext}_A^1(T,?)),\text{Ker}(\text{Hom}(T,?)))$ and $(\text{Ker}(?\otimes_BT),\text{Ker}(\text{Tor}_1^B(?,T)))$ of $\text{mod}-A$ and $\text{mod}-B$, respectively, where $B=\text{End}(T_A)$ is the endomorphism algebra and $\text{mod}$ denotes the category of finite dimensional modules. This theorem was extended in \cite{CF} by considering $A$ to be an arbitrary ring, $T$ to be a finitely presented $A$-module and the module categories over $A$ and $B$ to be those of all modules, not just
 the finitely presented ones. By definition, in all these papers tilting modules had projective dimension less or equal than one. A generalization of the concept to modules of finite projective dimension greater than one was given in \cite{M}, where the new tilting modules had a finite projective resolution with finitely generated terms (we will call these modules strongly finitely presented). The substitute of Brenner-Butler theorem in this context was a series of equivalences of full subcategories  $\text{Ker}(\text{Ext}_A^k(T,?))\stackrel{\cong}{\longleftrightarrow}\text{Ker}(\text{Tor}_k^B(?,T))$, for all integers $k\in\mathbb{Z}$. 

Up to here, even working over arbitrary rings, tilting modules were assumed to be strongly finitely presented. The historical development brought a new notion of tilting module, where finite projective dimension was still required but not the strongly finitely presented condition. This was first done in \cite{CT}, for modules of projective dimension less or equal than one, and in \cite{AC} for modules of  arbitrary finite projective dimension. Since then the most commonly used terminology reserves the term \emph{classical ($n$-)tilting module} for the strongly finitely presented modules of finite projective dimension (less or equal than $n$), while the term \emph{($n$-)tilting module} applies to all modules, not just the strongly finitely presented ones. 

Even from the beginning, tilting theory expanded from studying tilting modules to study tilting objects in more general abelian categories,  and this extension has played  a fundamental role in several branches of Mathematics, as Algebraic Geometry and Representation Theory of Groups and Algebras (see \cite{AHK}).  

Soon after their introduction, it became clear that tilting modules and tilting objects had to do with derived categories. In \cite{CPS} and \cite{H} (see also \cite{H2}), the authors proved
that any classical tilting module gives an equivalence  of bounded derived categories $\mathcal{D}^b(A)\stackrel{\cong}{\longrightarrow}\mathcal{D}^b(B)$. The general Morita theory for derived categories, developed immediately after by Rickard and Keller (see \cite{R} and \cite{K}),  showed that this equivalence actually extended to one between the corresponding unbounded derived categories. It was then the turn for nonclassical tilting modules to be related to derived categories. This was first done  implicitly in \cite{HRS}, within the study of quasi-tilted algebras. The authors proved that each torsion pair $\mathbf{t}=(\mathcal{T},\mathcal{F})$ in an abelian category $\mathcal{A}$ defined a t-structure in $\mathcal{D}^b(\mathcal{A})$ (and also in $\mathcal{D}(\mathcal{A})$), in the sense of \cite{BBD}, whose heart $\mathcal{H}$ admitted  $(\mathcal{X},\mathcal{Y}):=(\mathcal{F}[1],\mathcal{T})$ as a torsion pair (counterequivalent to $\mathbf{t}$). Then they showed that 
 if either $\mathcal{T}$ was a cogenerating or $\mathcal{F}$ was a generating class, then the inclusion $\mathcal{H}\hookrightarrow\mathcal{D}^b(\mathcal{A})$ extended to a triangulated equivalence $\mathcal{D}^b(\mathcal{H})\stackrel{\cong}{\longrightarrow}\mathcal{D}^b(\mathcal{A})$. This  result can be then applied when $\mathcal{A}=\text{Mod}-A$ and the torsion pair is either $(\text{Ker}(\text{Ext}_A^1(T,?)),\text{Ker}(\text{Hom}(T,?)))$, the one associated to any $1$-tilting module, or $(\text{Ker}(\text{Hom}_A(?,Q)),\text{Ker}(\text{Ext}_A^1(?,Q)))$, the one associated to a $1$-cotilting module (defined dually). Furthermore, the module $T$ is a projective generator of $\mathcal{H}$, which is a progenerator when $T$ is classical $1$-tilting (see \cite[III.4]{BR}). Therefore in this latter case $\mathcal{H}$ is equivalent to $\text{Mod}-B$ and, as a consequence, the equivalence of \cite{CPS} or \cite{H} for classical 1-tilting modules is just a particular case of the mentioned result of 
 \cite{HRS}. For another example of the relatioship between nonclassical tilting modules and derived equivalences, we refer to \cite{BMT}.

When $T$ is an arbitrary $n$-tilting module, we no longer have a torsion pair in $\text{Mod}-A$, although the pair $\tau_T=(T^{\perp_{>0}},T^{\perp_{<0}})$ in $\mathcal{D}(A)$ (or $\mathcal{D}^b(A)$) is still a t-structure (see, for example \cite{NSZM}). The mentioned result of \cite{HRS} naturally leads to two questions. The first one would ask whether the equivalence $\mathcal{D}^b(\mathcal{H})\cong\mathcal{D}^b(A)$ for $1$-tilting modules  also holds at the unbounded level. The second one would ask whether this equivalence (at the bounded or unbounded level) also holds when we replace $1$-tilting (resp. $1$-cotilting) by $n$-tilting ($n$-cotilting), for $n$ arbitrary, and $\mathcal{H}$ is the heart of $\tau_T$. The first question was answered in the affirmative by Xiao Wu Chen \cite{Chen}, for any $1$-tilting object in an arbitrary abelian category. As for  the second question, an affirmative answer in the case of a classical $n$-tilting object in a Grothendieck category, both in the bounded and unbonded
 ed situations,  follows from \cite{FMT}. When $T$ is a nonclassical tilting object, the second question has an affirmative answer at  the bounded level, in essentially any AB3 abelian category,  due to a stronger recent result of Psaroudakis and Vit\'oria (see \cite{PV}[Corollary 5.2]). The main goal of this short note is to prove that the second question  has an affirmative answer also at the unbounded level for any ($n$-)tilting object in essentially any AB3(=cocomplete) abelian category $\mathcal{A}$. More concretely, the following result and its dual will be direct consequences of our main theorem (see Theorem \ref{teor.tilting derived equivalence}):

\vspace*{0.3cm}
{\bf\noindent Proposition 2.3.} {\it Let  $\mathcal{A}$ be an abelian category such that its derived category $\mathcal{D}(\mathcal{A})$ has $Hom$ sets and arbitrary (small) coproducts, let $T$ be a tilting object in $\mathcal{A}$ 
and let $\mathcal{H}$ be the heart of the associated t-structure $\tau_T=(T^{\perp_{>0}},T^{\perp_{<0}})$ in $\mathcal{D}(\mathcal{A})$. Then there is a triangulated equivalence $\mathcal{D}(\mathcal{H})\stackrel{\cong}{\longrightarrow}\mathcal{D}(\mathcal{A})$ which is compatible with the inclusion functor $\mathcal{H}\hookrightarrow\mathcal{D}(\mathcal{A})$.
}

In this statement and in the rest of the paper, the compatibility with the inclusion $\mathcal{H}\hookrightarrow\mathcal{D}(\mathcal{A})$ means that the restriction of the equivalence to $\mathcal{H}$ is naturally isomorphic to the inclusion functor. 
\vspace*{0.3cm}

The organization of the paper goes as follows. In Section 1, we recall some basic fact about $t$-structures in triangulated categories and we prove the Tilting Theorem~\ref{teor.tilting derived equivalence}. Here the setup is very general. We consider a 
$t$-structure $\tau =(\mathcal{U},\mathcal{U}^\perp [1])$ in $\mathcal{D}(\mathcal{A})$ such that $\mathcal{D}^{\leq -n}(\mathcal{A})\subseteq\mathcal{U}\subseteq\mathcal{D}^{\leq 0}(\mathcal{A})$ and consider the class $\mathcal{Y}:=\mathcal{A}\cap\mathcal{H}_\tau$, where $\mathcal{H}_\tau$ is the heart of $\tau$. Theorem \ref{teor.tilting derived equivalence} states that if $\mathcal{Y}$ is cogenerating in $\mathcal{A}$, then there is a triangulated equivalence $\mathcal{D}(\mathcal{H}_\tau)\stackrel{\cong}{\longrightarrow}\mathcal{D}(\mathcal{A})$ which is compatible with the inclusion functor $\mathcal{H}_\tau\hookrightarrow\mathcal{D}(\mathcal{A})$.
In Section 2 we apply the Tilting Theorem to nonclassical tilting objects. Firstly, we give the definition of (nonclassical) tilting object in any abelian category as in the statement and a give characterizations suitable for our purposes. Both, definition and characterizations, are taken from  \cite{NSZM}. Then we show that the above Proposition~\ref{cor.tilting object Grothendieck category} and its dual (see Proposition \ref{cor.equivalence cotilting}, which includes a recent result of Stovicek about derived equivalences induced by $n$-cotilting modules) are direct consequences of the Tilting Theorem. Finally, as a byproduct, we prove that a ($n$-)tilting object as in the situation of Proposition~\ref{cor.tilting object Grothendieck category} is self-small if and only if it is classical tilting (Corollary \ref{cor.self-small tilting object}).  Except for the case $n=1$ (see \cite{CT}[Proposition 1.3]), this fact seems to be unknown even for tilting modules.

\section{Tilting Theorem}
Let $\mathcal{D}$ be an additive category. For any full subcategory $\mathcal{D}'$ of $\mathcal{D}$ we denote by 
$^\perp\mathcal{D}'$ the following full subcategory of $\mathcal{D}$: 
\[
^\perp\mathcal{D}':=\{X\in \mathcal{D} \,|\, \Hom_\mathcal{D}(X,Y)=0, \text{~for all~} Y\in\mathcal{D}'\},
\]
 and by $\mathcal{D}'^\perp$ the following full subcategory of $\mathcal{D}$: 
\[
\mathcal{D}'^\perp:=\{X\in \mathcal{D} \,|\, \Hom_\mathcal{D}(Y,X)=0, \text{~for all~} Y\in\mathcal{D}'\}.
\]

Let $\mathcal{D}$ be a triangulated category and denote its suspension functor by $?[1]$. If $\mathcal{U},\mathcal{V}$ are full subcategories of $\mathcal{D}$, then we denote by $\mathcal{U}\star \mathcal{V}$ the category of extensions of $\mathcal{V}$ by $\mathcal{U}$, that is, the full subcategory of $\mathcal{D}$ consisting of objects $X$ which may be included in a distinguished triangle $U\to X\to V \stackrel{+}\to\,$ in 
$\mathcal{D}$, with $U \in \mathcal{U}$ and $V \in \mathcal{V}$.

Recall that a {\it $t$-structure} in $\mathcal{D}$ is a pair $\tau=(\mathcal{U},\mathcal{V})$ of full additive subcategories of $\mathcal{D}$ which satisfy the following properties:
\begin{enumerate}
\item[\rm (i)]$\Hom_\mathcal{D}(U,V[-1]) = 0$, for all $U \in \mathcal{U}$ and $V \in \mathcal{V}$.
\item[\rm (ii)]$\mathcal{U}[1] \subseteq \mathcal{U}$.
\item[\rm (iii)]$\mathcal{D}=\mathcal{U}\star\mathcal{V}[-1]$, that is, for each object $X$ of $\mathcal{D}$, there is a distinguished triangle $U \to X \to V \stackrel{+}\to$ in $\mathcal{D}$, where 
$U \in \mathcal{U}$ and $V \in \mathcal{V}[-1]$.
\end{enumerate}
In such case one has $\mathcal{V}=\mathcal{U}^\perp[1]$ and $\mathcal{U}=\;^\perp(\mathcal{V}[-1])=\;^\perp(\mathcal{U}^\perp)$. Thus, we may write a $t$-structure as 
$\tau=(\mathcal{U},\mathcal{U}^\perp[1])$. The objects $U$ and $V$ in the above triangle are uniquely determined by $X$, up to isomorphism, and define functors $\tau_\mathcal{U}\colon \mathcal{D} \to \mathcal{U}$ and 
$\tau^{\mathcal{U}^\perp} \colon \mathcal{D} \to \mathcal{U}^\perp$ which are right and left adjoints to the respective inclusion functors. The full subcategory $\mathcal{H}_\tau = \mathcal{U} \cap \mathcal{V} = \mathcal{U} \cap \mathcal{U}^\perp[1]$ is called the {\it heart} of the $t$-structure and it is an abelian category, where the short exact sequences ``are'' the triangles in $\mathcal{D}$ with their three terms in $\mathcal{H}_\tau$. The assignments 
$X\leadsto(\tau_\mathcal{U} \circ \tau^{U^\perp[1]})(X)$ and 
$X \leadsto (\tau^{\mathcal{U}\perp[1]} \circ \tau_{\mathcal{U}})(X)$ define naturally isomorphic functors $\mathcal{D} \to \mathcal{H}$ which are cohomological (see~\cite{BBD}). In the sequel, we will use
the notation $\tau^{[a,b]}=\mathcal{U}[-b] \cap \mathcal{U}^\perp[1-a]$, where $a\leq b$ are integers. Given an abelian category $\mathcal{A}$, its derived category $\mathcal{D}(\mathcal{A})$ is a triangulated category which admits a 
$t$-structure $\delta=(\mathcal{D}(\mathcal{A})^{\leq0},\mathcal{D}(\mathcal{A})^{\geq0})$, called the {\it natural $t$-structure}, where 
$\mathcal{D}(\mathcal{A})^{\leq0}$ (resp. $\mathcal{D}(\mathcal{A})^{\geq0}$) is the full subcategory of complexes without cohomology in positive (resp. negative) degrees. The heart of $\delta$ is (naturally equivalent to) $\mathcal{A}$.

For the rest of this section, we make the following assumptions and set the following notation: 

\begin{setup}\label{setup1}$\mathcal{A}$ is any abelian category such that $\mathcal{D}(\mathcal{A})$ has $Hom$ sets. We denote by $\delta=(\mathcal{D}(\mathcal{A})^{\leq0},\mathcal{D}(\mathcal{A})^{>0}[1])$ the natural $t$-structure on $\mathcal{D}(\mathcal{A})$ and by $\tau=(\mathcal{U},\mathcal{U}^\perp[1])$ another $t$-structure on 
$\mathcal{D}(\mathcal{A})$, whose heart is $\mathcal{H}_\tau$. We shall assume that 
\[\mathcal{D}(\mathcal{A})^{\leq-n}\subseteq \mathcal{U} \subseteq \mathcal{D}(\mathcal{A})^{\leq0} \quad (\text{equivalently,}\quad
\mathcal{D}(\mathcal{A})^{>0}\subseteq \mathcal{U}^\perp \subseteq \mathcal{D}(\mathcal{A})^{>{-n}}). 
\]
Consider the class $\mathcal{Y}=\mathcal{A}\cap\mathcal{H}_\tau$ in the sequel. We denote by 
$\mathcal{K}_{ac}^{\mathcal{A}}(\mathcal{Y})$ (resp. $\mathcal{K}_{ac}^{\mathcal{H}_\tau}(\mathcal{Y})$) the additive full subcategory of the homotopy category $\mathcal{K}(\mathcal{Y})$ of cochain complexes on $\mathcal{Y}$ whose objects are the complexes of $\mathcal{K}(\mathcal{Y})$ which are acyclic in 
$\mathcal{A}$ (resp. in $\mathcal{H}_\tau$).
\end{setup}

\begin{lemma}\label{cohom}
Let $s$ be a natural number and let $Y^\bullet$ be the complex
\[
\xymatrix@-2pt{\dots \ar[r] &0 \ar[r] & Y^{-s} \ar^{d^{-s}}[r]& Y^{-s+1} \ar[r]& \dots \ar[r]& Y^{-1} \ar^{d^{-1}}[r]& Y^{0} \ar[r]& 0\ar[r]& \dots}
\]
where $Y_i$ is an object of $\mathcal{Y}$, for every $i=-s, \dots, -1, 0$. Then $Y^\bullet \in \delta^{[-s,0]} \cap \tau^{[-s,0]}$.
\end{lemma}
\begin{proof}
The complex $Y^\bullet$ belongs to the following full subcategory of $\mathcal{D}(\mathcal{A})$:
\[
(\mathcal{D}(\mathcal{A})^{>0}[1] \cap \mathcal{U})[0] \star (\mathcal{D}(\mathcal{A})^{>0}[1] \cap \mathcal{U})[1] \star \cdots \star (\mathcal{D}(\mathcal{A})^{>0}[1] \cap \mathcal{U})[s-1] \star (\mathcal{D}(\mathcal{A})^{>0}[1] \cap \mathcal{U})[s],
\]
The latter subcategory is contained in $\mathcal{D}(\mathcal{A})^{>0}[1+s] \cap \mathcal U$. By the assumption made in Setup~\ref{setup1}, 
$\mathcal U \subseteq \mathcal{D}(\mathcal{A})^{\leq0}$ and $\mathcal{D}(\mathcal{A})^{>0}[1+s]\subseteq \mathcal{U}^\perp[1+s]$. From this we conclude that $Y^\bullet$ belongs to
\[
(\mathcal{D}(\mathcal{A})^{\leq0}\cap \mathcal{D}(\mathcal{A})^{>0}[1+s])\cap (\mathcal{U}\cap \mathcal{U}^\perp[1+s])=\delta^{[-s,0]} \cap \tau^{[-s,0]}.
\]
\end{proof}

\begin{lemma}\label{lemma0}
Let $Y^{-n}\stackrel{f_{-n}}\to Y^{-n+1}\to \dots \to Y^{-1}\stackrel{f_{-1}}\to Y^0$ be a sequence of morphisms in 
$\mathcal{Y}$ such that $f_{-k}\circ f_{-k-1}=0$, for every $k=1,\dots, n-1$, where $n$ is the natural number such that $\mathcal{D}(\mathcal{A})^{\leq-n}\subseteq \mathcal{U} \subseteq \mathcal{D}(\mathcal{A})^{\leq0}$. The following assertions hold.
\begin{enumerate}
\item If the sequence is exact in $\mathcal{H}_\tau$, then ${\rm Ker}_{\mathcal{H}_\tau}f_{-n} \in \mathcal{Y}$.
\item If the sequence is exact in $\mathcal{A}$, then ${\rm Coker}_\mathcal{A}f_{-1} \in \mathcal{Y}$.
\end{enumerate}
\end{lemma}
\begin{proof}
Let $Y^\bullet$ be the complex in $\mathcal{D}(\mathcal{A})$ obtained from the sequence 
$Y^{-n}\stackrel{f_{-n}}\to Y^{-n+1}\to \dots \to Y^{-1}\stackrel{f_{-1}}\to Y^0$ by completing with zeros on both sides. By Lemma~\ref{cohom}, $Y^\bullet \in \delta^{[-n,0]} \cap \tau^{[-n,0]}$.

1) For every $k=0,1,\dots,n$ there is a distinguished triangle in $\mathcal{D}(\mathcal{A})$
\[
\xymatrix{{\rm Ker}_{\mathcal{H}_\tau}(f_{-k})\ar[r]& Y^{-k}\ar[r]& {\rm Ker}_{\mathcal{H}_\tau}(f_{-k+1})\ar^-{+}[r]&}
\]
(with the convention that $f_0 \colon Y^0 \to L:={\rm Coker}_{\mathcal{H}_\tau}(f_{-1})=\text{Ker}_{\mathcal{H}_\tau}(f_1)$ is the cokernel morphism of $f_{-1}$). By an iterated use of the octahedral axiom, we obtain the following distinguished triangle in $\mathcal{D}(\mathcal{A})$:
\[
\xymatrix{{\rm Ker}_{\mathcal{H}_\tau}(f_{-n})[n]\ar[r]& Y^\bullet \ar[r]& L\ar^-{+}[r]&}
\]
thus by shifting:
\[
\xymatrix{{\rm Ker}_{\mathcal{H}_\tau}(f_{-n})\ar[r]& Y^\bullet[-n] \ar[r]& L[-n]\ar^-{+}[r]&}
\]
Now, $Y^\bullet[-n] \in \mathcal{D}(\mathcal{A})^{>0}[1]$ and 
$L[-n] \in \mathcal{H}_\tau[-n]\subseteq \mathcal{D}(\mathcal{A})^{>0}[1]$. Thus 
${\rm Ker}_{\mathcal{H}_\tau}(f_{-n}) \in \mathcal{D}(\mathcal{A})^{>0}[1]$. Therefore 
${\rm Ker}_{\mathcal{H}_\tau}(f_{-n}) \in \mathcal{D}(\mathcal{A})^{\leq0} \cap \mathcal{D}(\mathcal{A})^{>0}[1]=\mathcal{A}$ since 
$\mathcal{H}_\tau \subseteq \mathcal{D}(\mathcal{A})^{\leq0}$. It follows that 
${\rm Ker}_{\mathcal{H}_\tau}(f_{-n}) \in \mathcal{A}\cap \mathcal{H}_\tau=\mathcal{Y}$.

2) This is the symmetric argument of the previous one, using the fact that $Y^\bullet \in \tau^{[-n,0]}$.
\end{proof}

\begin{proposition}\label{acyclic}
Let $Y^\bullet$ be a complex of objects in $\mathcal{Y}$. The following are equivalent:
\begin{enumerate}
\item $Y^\bullet$ is acyclic in $\mathcal{A}$.
\item $Y^\bullet$ is acyclic in $\mathcal{H}_\tau$.
\end{enumerate}
\end{proposition}
\begin{proof}
$1)\Longrightarrow 2)$ The complex $Y^\bullet$ is obtained by gluing the following short exact sequences in $\mathcal{A}$:
\begin{equation}\label{seq}
\begin{gathered}
\xymatrix{0\ar[r]& {\rm Coker}_\mathcal{A}(d^{k-2}_{Y^\bullet})\ar[r]& Y^k\ar[r]& {\rm Coker}_\mathcal{A}(d^{k-1}_{Y^\bullet})\ar[r]& 0}, \quad k\in \mathbb{Z}.
\end{gathered}
\end{equation}
Now the sequences $Y^{k-n} \stackrel{d^{k-n}_{Y^\bullet}}\to Y^{k-n+1}\to \dots \to Y^{k}$, $k\in \mathbb{Z}$, satisfy the hypothesis of Lemma~\ref{lemma0}, 2), hence 
${\rm Coker}_\mathcal{A}(d^{k-1}_{Y^\bullet}) \in \mathcal Y$, for all $k\in\mathbb Z$. Then we have distinguished triangles in 
$\mathcal{D}(\mathcal{A})$
\[
\xymatrix{ {\rm Coker}_\mathcal{A}(d^{k-2}_{Y^\bullet})\ar[r]& Y^k\ar[r]& {\rm Coker}_\mathcal{A}(d^{k-1}_{Y^\bullet})\ar[r]^-+&}, \quad k\in \mathbb{Z}.
\]
with all the terms in $\mathcal{H}_\tau$. From this we can conclude that the sequences~(\ref{seq}) are short exact in 
$\mathcal{H}_\tau$ and hence that $Y^\bullet$ is acyclic in $\mathcal{H}_\tau$.

$2)\Longrightarrow 1)$ This is just a symmetric argument using Lemma~\ref{lemma0}, 1).
\end{proof}

In order to construct resolutions of complexes, we will need the following two auxiliary results and their duals, which hold true since the opposite category of an abelian category is again abelian.

\begin{lemma}\cite[Lemma~13.2.1]{KaSc}\label{ks1}
Let $\mathcal{A}$ be an abelian category, let $\mathcal I$ be an additive full subcategory of $\mathcal A$ and let 
$X \in \text{C}^{\geq a}(\mathcal A)$ for some $a \in \mathbb{Z}$. Assume that $\mathcal I$ is cogenerating in $\mathcal A$. Then there exist $Y \in \text{C}^{\geq a}(\mathcal I)$ and a quasi-isomorphism $X \to Y$.
\end{lemma}

\begin{proposition}\cite[Proposition~13.2.6]{KaSc}\label{ks2}
Let $\mathcal{A}$ be an abelian category and let $\mathcal I$ be an additive full subcategory of $\mathcal A$. Assume that:
\begin{itemize}
\item[\rm (i)]$\mathcal I$ is cogenerating,
\item[\rm (ii)]there exists a non-negative integer $d$ such that, for any exact sequence $Y^d \to \cdots \to Y^1\to Y \to 0$, with 
$Y^j \in \mathcal I$ for $1 \leq j\leq d$, we have $Y\in \mathcal I$.
\end{itemize}
Then for any complex $X \in \text{C}(\mathcal A)$, there exist $Y \in \text{C}(\mathcal I)$ and a quasi-isomorphism $X \to Y$. In particular, there is an equivalence of triangulated categories 
\[
\text{K}(\mathcal I)/\mathcal{K}_{ac}^{\mathcal{A}}(\mathcal I) \stackrel{\cong}\to \mathcal{D}(\mathcal A),
\]
where $\mathcal{K}_{ac}^{\mathcal{A}}(\mathcal I)$ is the class of complexes on $\mathcal I$ which are acyclic in $\mathcal A$.
\end{proposition}

\begin{lemma}\label{tiltvscotitl}
If the class $\mathcal{Y}=\mathcal{A}\cap\mathcal{H}_\tau$ is cogenerating
in $\mathcal{A}$, then
$\mathcal{Y}$ is generating in $\mathcal{H}_\tau$.
\end{lemma}
\begin{proof}
Let us suppose that $\mathcal{Y}$ cogenerates $\mathcal{A}$ and let us prove that 
$\mathcal{Y}$ generates $\mathcal{H}_\tau$.
Let $B \in \mathcal{H}_\tau$. First we show that we can represent the object $B$ as a complex
\[Y^\bullet:=
\cdots \to 0 \to Y^{-n} \stackrel{d^{-n}}\to Y^{-n+1}\stackrel{d^{-n+1}}\to\cdots 
\stackrel{d^{-2}}\to Y^{-1} \stackrel{f}\to A \to 0\to \cdots
\]
whose terms  $Y_i$ lie in $\mathcal{Y}$ for any $-n\leq i\leq -1$, $A\in\mathcal{A}$ and it is placed in degree zero. Indeed, since $\mathcal{H}_\tau\subseteq \delta^{[-n, 0]}$, Lemma~\ref{ks1} gives that $B$ is quasi-isomorphic to a complex 
$...\rightarrow 0\rightarrow Y^{-n}\rightarrow ...\rightarrow Y^0\rightarrow Y^1\rightarrow ...$, with $Y^i$ in $\mathcal{Y}$ for all $i\geq -n$. Then using the classical truncation at $0$ by taking $A={\rm Ker}_\mathcal{A}(Y^0\rightarrow Y^1)$, gives a complex 
$... \rightarrow 0\rightarrow Y^{-n}\rightarrow ...\rightarrow Y^{-1}\rightarrow A\rightarrow 0 \rightarrow...$, which is clearly quasi-isomorphic to $Y^\bullet$, and hence isomorphic to $B$ in $\mathcal{D}(\mathcal{A})$.

Now let us consider the following distinguished triangle in $\mathcal{D}(\mathcal{A})$ (see on the left in the vertical sense the distinguished triangle):
\[
\xymatrix{
C\ar[d] & 0\ar[r]\ar[d] & Y^{-n}\ar[r]\ar[d]	& Y^{-n+1}\ar[r]\ar[d]		&	\cdots\ar[r]\ar[d]&
Y^{-2}\ar[r]\ar[d] &
Y^{-1}\ar[d]^{f} \\
A[0]\ar[d] &0\ar[r]\ar[d] & 0\ar[r]\ar[d] & 0\ar[r] \ar[d] &0\ar[r] \ar[d] &  0\ar[r] \ar[d] &A\ar[d]\\
B &  Y^{-n}\ar[r] & Y^{-n+1}\ar[r]&	\cdots\ar[r] & \cdots\ar[r] &Y^{-1}\ar[r] &A. \\
}
\]
By Lemma~\ref{cohom}, we have $C\in\delta^{[-n,0]}\cap\tau^{[-n,0]}$.
Since $B\in\mathcal{H}_\tau$ we deduce that $A\in \mathcal{A}\cap\mathcal{U}=\mathcal{Y}$.
Now $A,B\in\mathcal{H}_\tau$ hence $C\in\mathcal{U}^\perp[1]\cap\tau^{[-n,0]}=\mathcal{H}_\tau$.
We can conclude that $\mathcal{Y}$ generates $\mathcal{H}_\tau$  since for any $B\in\mathcal{H}_\tau$ we have found
$A\in\mathcal{Y}$ and the previous distinguished triangle proves that
$0\to C \to A\to B\to 0$ is an exact sequence in $\mathcal{H}_\tau$.
\end{proof}


\begin{theorem}[Tilting Theorem] \label{teor.tilting derived equivalence}
Let $\mathcal{A}$ be an abelian category such that $\mathcal{D}(\mathcal{A})$ has $Hom$ sets and let $\tau =(\mathcal{U},\mathcal{U}^\perp [1])$  be a t-structure in $\mathcal{D}(\mathcal{A})$ such that $\mathcal{D}^{\leq -n}(\mathcal{A})\subseteq\mathcal{U}\subseteq\mathcal{D}^{\leq 0}(\mathcal{A})$, for some $n\in\mathbb{N}$,  and let $\mathcal{H}_\tau$ be its heart. If  the class $\mathcal{Y}=\mathcal{A}\cap\mathcal{H}_\tau$ cogenerates $\mathcal{A}$, then there is a triangulated equivalence 
\[
\mathcal{D}(\mathcal{H}_\tau)\stackrel{\cong}{\longrightarrow}\mathcal{D}(\mathcal{A})
\]
which is compatible with the inclusion functor $\mathcal{H}_\tau\hookrightarrow\mathcal{D}(\mathcal{A})$.
\end{theorem}
\begin{proof}
By assumption, the full subcategory $\mathcal{Y}$ is cogenerating in $\mathcal{A}$ and by Lemma~\ref{lemma0},1) $\mathcal{Y}$ satisfies assumption (ii) of Proposition~\ref{ks2}. 
Therefore 
we can apply the latter result, which gives that the canonical composition of functors 
$\mathcal{K}(\mathcal{Y}) \xhookrightarrow{} \mathcal{K}(\mathcal{A}) \xhookrightarrow{} \mathcal{D}(\mathcal{A})$ induces a triangulated equivalence
$\mathcal{K}(\mathcal{Y})/\mathcal{K}_{ac}^{\mathcal{A}}(\mathcal{Y}) \stackrel{\cong}\to \mathcal{D}(\mathcal{A})$. On the other hand, by Lemma~\ref{tiltvscotitl} the full subcategory $\mathcal{Y}$ is generating in $\mathcal{H}_\tau$ and by Lemma~\ref{lemma0}, 2) 
$\mathcal{Y}$ satisfies the assumptions of the dual of Proposition~\ref{ks2}. 
Therefore by the dual of 
Proposition~\ref{ks2}, the canonical composition of functors 
$\mathcal{K}(\mathcal{Y}) \xhookrightarrow{} \mathcal{K}(\mathcal{H}_\tau) \xhookrightarrow{} \mathcal{D}(\mathcal{H}_\tau)$ induces a triangulated equivalence 
$\mathcal{K}(\mathcal{Y})/\mathcal{K}_{ac}^{\mathcal{H}_\tau}(\mathcal{Y}) \stackrel{\cong}\to \mathcal{D}(\mathcal{H}_\tau)$. By Proposition~\ref{acyclic}, $\mathcal{K}_{ac}^{\mathcal{A}}(\mathcal{Y})=\mathcal{K}_{ac}^{\mathcal{H}_\tau}(\mathcal{Y})=:\mathcal{K}_{ac}{(\mathcal{Y})}$. Hence we obtain the following diagram
\[
\xymatrix{&\mathcal{K}(\mathcal{Y})/\mathcal{K}_{ac}(\mathcal{Y}) \ar[dl]_\cong \ar[dr]^\cong &\\
\mathcal{D}(\mathcal{H}_\tau) \ar^{\cong}[rr]&& \mathcal{D}(\mathcal{A})}
\]

Finally, consider any object $X^\bullet$ of $\mathcal{H}_\tau$ and denote by $F$ the just defined equivalence $\mathcal{D}(\mathcal{H}_\tau)\stackrel{\cong}{\longrightarrow}\mathcal{D}(\mathcal{A})$. By the dual of Proposition \ref{ks2}, we have a quasi-isomorphism $Y^\bullet\longrightarrow X^\bullet$ in $\mathcal{K}(\mathcal{H}_\tau)$, whence an isomorphism in $\mathcal{D}(\mathcal{H}_\tau)$, where $Y^\bullet\in\mathcal{K}(\mathcal{Y})$. By the above commutative diagram, we have an isomorphism $F(X^\bullet )\cong F(Y^\bullet )\cong Y^\bullet$ in $\mathcal{D}(\mathcal{A})$, which shows the compatibility of $F$ with the inclusion functor $\mathcal{H}_\tau\hookrightarrow\mathcal{D}(\mathcal{A})$. 
\end{proof}

\begin{remark}\label{dualtilt}
Lemma~\ref{tiltvscotitl} can be dualized in the following way: let $\mathcal{A}$ be an abelian category such that $\mathcal{D}(\mathcal{A})$ has $Hom$ sets and let $\tau =(\mathcal{U},\mathcal{U}^\perp [1])$  be a $t$-structure in $\mathcal{D}(\mathcal{A})$ such that $\mathcal{D}^{\leq 0}(\mathcal{A})\subseteq\mathcal{U}\subseteq\mathcal{D}^{\leq n}(\mathcal{A})$, for some $n\in\mathbb{N}$,  and let $\mathcal{H}_\tau$ be its heart. If  the class $\mathcal{Y}=\mathcal{A}\cap\mathcal{H}_\tau$ is generating in $\mathcal{A}$, then $\mathcal Y$ is cogenerating in $\mathcal{H}_\tau$. The proof of this statement follows by using the dual arguments of those presented in the proof of Lemma~\ref{tiltvscotitl}. Under these assumptions, Theorem~\ref{teor.tilting derived equivalence} yields a triangle equivalence 
$\mathcal{D}(\mathcal{H}_\tau )\stackrel{\cong}{\longrightarrow}\mathcal{D}(\mathcal{A})$ which is compatible with the inclusion 
$\mathcal{H}_\tau\hookrightarrow\mathcal{D}(\mathcal{A})$. 
\end{remark}

\section{Applications to nonclassical tilting objects}

In this section we apply the Tilting Theorem~\ref{teor.tilting derived equivalence} to the case of nonclassical tilting objects.

The following definition appears in~\cite{NSZM}. 

\begin{definition}\cite{NSZM}\label{ntilting}
Let $\mathcal{A}$ be an abelian category whose derived category has 
$Hom$ sets and arbitrary (small) coproducts.
Let $n$ be a natural number. An object $T$ of $\mathcal{A}$ is called {\it $n$-tilting} when the following conditions hold:
\begin{enumerate}
\item[\rm T0] For each set $I$, the object $T^{(I)}$ is the coproduct of $I$ copies of $T$ both in $\mathcal{A}$ and 
$\mathcal{D}(\mathcal{A})$;
\item[\rm T1] $\Ext_\mathcal{A}^k(T,T^{(I)})=0$, for all integers $k>0$ and all sets I;
\item[\rm T2] The projective dimension of $T$ is $\leq n$, that is, $\Ext_\mathcal{A}^k(T,N)=0$ for all integers $k>n$ and all objects $N$ of 
$\mathcal{A}$;
\item[\rm T3] There is a generating class $\mathcal{G}$ of $\mathcal{A}$ such that, for each $G \in \mathcal{G}$, there is an exact sequence
$0\to G\to T^0\to T^1\to \dots \to T^n\to 0$, where all the $T^k$ are in $\Add(T)$.
\end{enumerate}
$T$ is called a {\it tilting} object when it is $n$-tilting, for some natural number $n$.  A {\it classical ($n$-)tilting} object is a ($n$-)tilting object which is compact as an object of $\mathcal{D}(\mathcal{A})$. 
\end{definition} 

As proved in \cite{NSZM}, if $\mathcal{A}$ be an abelian category whose derived category has 
$Hom$ sets and arbitrary (small) coproducts, as in Definition~\ref{ntilting}, then $\mathcal{A}$ is automatically cocomplete.
When $\mathcal{A}$ is an abelian category whose derived category has $Hom$ sets and arbitrary (small) products, the dual concepts of {\it $n$-cotilting}, {\it cotilting} and {\it classical ($n$-)cotilting} object appear naturally.

Recall from~\cite{PV} that a {\it silting} object in $\mathcal{D}(\mathcal{A})$ is a complex $T^\bullet$ such that the pair 
$({T^\bullet}^{\perp_{>0}}, {T^\bullet}^{\perp_{<0}})$ is a $t$-structure in $\mathcal{D}(\mathcal{A})$ and $T^\bullet\in {T^\bullet }^{\perp_{>0}}$. Here ${T^\bullet}^{\perp_{>0}}$ (resp. ${T^\bullet}^{\perp_{<0}}$) is the full subcategory of $\mathcal{D}(\mathcal{A})$ consisting of those $Y^\bullet$ such that $\text{Hom}_{\mathcal{D}(\mathcal{A})}(T^\bullet ,Y^\bullet [k])=0$, for all $k>0$ (resp. $k<0$).

 In the next result, when $m\in\mathbb{N}$, the subcategory of objects $Y$ of $\mathcal{A}$ which admit an exact sequence $T^{-m}\longrightarrow...\longrightarrow T^{-1}\longrightarrow T^0\longrightarrow Y\rightarrow 0$, with all the $T^{-k}$ in $\text{Add}(T)$, will be denoted by $\text{Pres}^m(T)$.

\begin{theorem}{\rm\cite{NSZM}}
Let $\mathcal{A}$ be an abelian category such that its derived category $\mathcal{D}(\mathcal{A})$ has $Hom$ sets and arbitrary (small) coproducts,  and let $T$ be an object of $\mathcal{A}$ such that the coproduct $T^{(I)}$ is the same in $\mathcal{A}$ and in 
$\mathcal{D}(\mathcal{A})$. The following are equivalent:
\begin{enumerate}
\item $T$ is a tilting object in $\mathcal{A}$.
\item $T$ has finite projective dimension and is a silting object of $\mathcal{D}(\mathcal{A})$.
\item If $\mathcal{Y}=\bigcap_{k>0}{\rm Ker}(\Ext^k_\mathcal{A}(T,?))$, then:
\begin{enumerate}
\item[\rm (a)]$\mathcal{Y}={\rm Pres}^m(T)$, for some integer $m \in \mathbb{N}$;
\item[\rm (b)]
$\mathcal{Y}$ is a cogenerating class of $\mathcal{A}$.
\end{enumerate}
\end{enumerate}
Moreover, when one of the previous assertion holds, the pair $\tau_T =(T^{\perp_{>0}},T^{\perp_{>0}})$ is a t-structure in $\mathcal{D}(\mathcal{A})$ and $\mathcal{D}(\mathcal{A})^{\leq-n}\subseteq T^{\perp_{>0}} \subseteq \mathcal{D}(\mathcal{A})^{\leq0}$. 
\end{theorem}


The main application of the Tilting Theorem~\ref{teor.tilting derived equivalence} 
is the following proposition, which generalizes the derived equivalence induced by a classical tilting object (at the level of unbounded derived categories) to the non-classical case. 

\begin{proposition} \label{cor.tilting object Grothendieck category}
Let  $\mathcal{A}$ be an abelian category whose derived category has $Hom$ sets and arbitrary (small) coproducts, let $T$ be a tilting object in $\mathcal{A}$ 
and let $\mathcal{H}_T$ be the heart of the associated t-structure $\tau_T=(T^{\perp_{>0}},T^{\perp_{<0}})$ in $\mathcal{D}(\mathcal{A})$. Then there is a triangulated equivalence $\mathcal{D}(\mathcal{H}_T)\stackrel{\cong}{\longrightarrow}\mathcal{D}(\mathcal{A})$ which is compatible with the inclusion functor $\mathcal{H}_T\hookrightarrow\mathcal{D}(\mathcal{A})$. 
\end{proposition}
\begin{proof}
The proof follows directly from 
Theorems \cite{NSZM} and Theorem~\ref{teor.tilting derived equivalence}, after we realize that $\mathcal{A}\cap\mathcal{H}_T=\mathcal{A}\cap T^{\perp_{>0}}=\bigcap_{k>0}{\rm Ker}(\Ext^k_\mathcal{A}(T,?))$. 
\end{proof}

It is worth stating explicitly the dual of the last proposition, which is known when $\mathcal{A}$ is a module category (see \cite{Sto2}[Theorems 4.5 and 5.21]).

\begin{proposition} \label{cor.equivalence cotilting}
Let $\mathcal{A}$ be an abelian category such that $\mathcal{D}(\mathcal{A})$ has $Hom$ sets and arbitrary (small) products, let $Q$ be a cotilting object in $\mathcal{A}$ 
and let $\mathcal{H}_Q$ be the heart of the associated t-structure $({}^{\perp_{<0}}Q,{}^{\perp_{>0}}Q)$ in $\mathcal{D}(\mathcal{A})$. Then there is a triangulated equivalence $\mathcal{D}(\mathcal{H}_Q)\stackrel{\cong}{\longrightarrow}\mathcal{D}(\mathcal{A})$ which is compatible with the inclusion $\mathcal{H}_Q\hookrightarrow\mathcal{D}(\mathcal{A})$.
\end{proposition}

Recall that an object $A$ of $\mathcal{A}$ is \emph{self-small} when the canonical morphism $\text{Hom}_\mathcal{A}(A,A)^{(I)}\longrightarrow\text{Hom}_\mathcal{A}(A,A^{(I)})$ is an isomorphism, for all sets $I$. The following result seems to be unknown even when $\mathcal{A}$ is a module category, and is an easy consequence of Theorem \ref{teor.tilting derived equivalence}.

\begin{corollary} \label{cor.self-small tilting object}
Let $\mathcal{A}$ be an abelian category whose derived category has $Hom$ sets and arbitrary (small) coproducts and let $T$ be a tilting object of $\mathcal{A}$. 
The following assertions are equivalent:

\begin{enumerate}
\item $T$ is a self-small object of $\mathcal{A}$.
\item $T$ is a classical tilting object of $\mathcal{A}$.
\end{enumerate} 
\end{corollary}
\begin{proof}
$2)\Longrightarrow 1)$ On the one hand, by definition of tilting object, the coproduct $T^{(I)}$ is the same in $\mathcal{A}$ and in $\mathcal{D}(\mathcal{A})$. On the other hand, since $T$ is classical by assumption, then it is a compact object of 
$\mathcal{D}(\mathcal{A})$. Hence, for any set I, we have a chain of isomorphisms 
\[\begin{array}{lcl}
\text{Hom}_\mathcal{A}(T,T)^{(I)} &\cong& \text{Hom}_{\mathcal{D}(\mathcal{A})}(T,T)^{(I)}  \\
     &\cong&\text{Hom}_{\mathcal{D}(\mathcal{A})}(T,T^{(I)})  \\
     &\cong& \text{Hom}_\mathcal{A}(T,T^{(I)}),
\end{array}\]
whose composition is the canonical morphism $\Hom_\mathcal{A}(T,T)^{(I)}\longrightarrow \Hom_\mathcal{A}(T,T^{(I)})$.

$1)\Longrightarrow 2)$ By \cite{NSZM}, we know that the heart $\mathcal{H}_T$ of the associated t-structure $\tau =(T^{\perp_{>0}},T^{\perp_{<0}})$ is a module category, with $\tilde{H}(T)=T$ as a progenerator. In particular, $T$ is a compact generator of $\mathcal{D}(\mathcal{H}_T)$.  If now $F:\mathcal{D}(\mathcal{H}_T)\stackrel{\cong}{\longrightarrow}\mathcal{D}(\mathcal{A})$ is a triangulated equivalence which is compatible with the inclusion functor $\mathcal{H}_T\hookrightarrow\mathcal{D}(\mathcal{A})$, then $F(T)\cong T$ is compact in $\mathcal{D}(\mathcal{A})$. 
\end{proof}


\end{document}